\documentclass{amsart}
\usepackage{pdfpages}
\usepackage{dirtytalk}
\usepackage{hyperref}
\newtheorem{theorem}{Theorem}[section]

\newtheorem{corollary}[theorem]{Corollary}
\theoremstyle{definition}
\newtheorem{definition}[theorem]{Definition}

\newtheorem{proposition}{Proposition}
\theoremstyle{remark}
\newtheorem{remark}[theorem]{Remark}

\numberwithin{equation}{section}

\begin{document}

\title{Geometric invariants of normal curves under conformal transformation in $\mathbb{E}^3$}

\author{Mohamd Saleem Lone}

\address{International Centre for Theoretical Sciences, Tata Institute of Fundamental Research, 560089, India}
\email{saleemraja2008@gmail.com (or) mohamdsaleem.lone@icts.res.in}


%    General info

%\date{January 1, 2001 and, in revised form, June 22, 2001.}



\begin{abstract}
 In this paper, we investigate the geometric invariant properties of  a normal curve on a smooth immersed surface under conformal transformation. We obtain an invariant-sufficient condition for the conformal image of a normal curve. We also find the deviations of normal and tangential components of the normal curve under the same motion. The results in \cite{9} are claimed as special cases of this paper.
\end{abstract}
\dedicatory{Dedicated to Prof. B.-Y. Chen}

\maketitle
\section{Introduction} 
The study of smooth maps is an important field of study in differential geometry. There are multiple ways of classifying motions, albeit we will focus on those which preserves certain geometric properties. Depending upon the invariant nature of the mean$(H)$ and the Gaussian curvatures$(K)$, we broadly classify the transformations in the following three equivalence classes: isometric, conformal and non-conformal or general motion. Isometry preserves lengths as well as the angles between the curves on the surfaces. In the language of geometry, isometry keeps the Gaussian curvature invariant and the mean curvature is altered. For example, we can easily find an isometry between  catenoid and a helicoid implying that they have the same $K$ but different $H$. Roughly speaking, diffeomorphisms and isometries define one class, however, when we have to study the problems associated with analytic functions of complex variables, we need a generalized class of transformations, known as conformal motions. In this case, the angle of intersection of any arbitrary pair of intersection arcs on the surface is invariant, while as the distances may not be. Conformal maps are very important in cartography. The simplest example of such a conformal transformation is the stereographic projection of a sphere onto a plane. This property of conformal maps was first used by
 Gerardus Mercator to form the first angle preserving map, commonly known as Mercator's world map. Recently in 2018, Bobenko and Gunn published an animated movie(must-watch) with the springer videoMATH  on conformal maps \cite{1}. Finally, in case of general motions, neither angles nor distances are preserved between any intersecting pairs of curves on a surface. It is to be noted that the usage of term motion, transformation or map stands for the same. 
 
Let $\mathcal{S}$ and $\tilde{{\mathcal{S}}}$ be two smooth immersed surfaces in $\mathbb{E}^3$ and $ \mathcal{J}: {\mathcal{S}}\rightarrow {\tilde{\mathcal{S}}}$ be a smooth map. Throughout this paper, the quantities associated with $\tilde{{\mathcal{S}}}$  will be deonted by 
$"\sim"$. A necessary and sufficient condition for $ \mathcal{J}$ to be conformal is that the first fundamental form quantities are proportional. In other words the area elements of $\mathcal{S}$ and $\tilde{\mathcal{S}}$ are proportional to a differentiable function(factor) commonly known as dilation function denoted by $\zeta(u,v)$. The conformal transformation is a generalized class of certain motions in the following way  \cite{6}:
\begin{itemize}
\item If $\zeta(u,v) \equiv c,$ where $c$ is a constant with $c\neq \{0,1\}$, then $ \mathcal{J}$ is called as homothetic transformation.
\item If $\zeta(u,v) \equiv 1,$ then $ \mathcal{J}$  becomes isometry.
\end{itemize}
 Let $V$ of a neighborhood  of an arbitrary point $p \in \mathcal{S}$ and 
 \begin{equation}\label{0}
  \mathcal{J} :V \subset {\mathcal{S}} \rightarrow \tilde{V} \subset \tilde{{\mathcal{S}}}
  \end{equation} be a diffeomorphism, where $\tilde{V}$ is an open neighborhood of $\mathcal{J}(p)$. Then $\mathcal{J}$ is said to be a local isometry if for all  $y_1,y_2 \in T_p({\mathcal{S}})$, we have
$$\langle y_1, y_2 \rangle_p =\langle d \mathcal{J}_p(y_1),d \mathcal{J}_p(y_2) \rangle_{ \mathcal{J}(p)}.$$
If for all $p \in \mathcal{S}$, in addition to diffeomorphism $\mathcal{J}$ is a bijection, then $\mathcal{J}$ is a global isometry. In such a case $\mathcal{S}$ and $\tilde{{\mathcal{S}}}$ are said to isometric(globally).

Let $\mathcal{E},\mathcal{F},\mathcal{G}$ and $\tilde{\mathcal{E}},\tilde{\mathcal{F}},\tilde{\mathcal{G}}$ are the first fundamental form coefficients of ${\mathcal{S}}$ and $\tilde{{\mathcal{S}}}$, respectively. A necessary sufficient condition for $\mathcal{S}$ and $\tilde{{\mathcal{S}}}$ to be isometric is that the first fundamental form coefficients are invariant, i.e., 
\begin{equation*} \mathcal{E}=\tilde{\mathcal{E}},\quad \mathcal{F}=\tilde{\mathcal{F}},\quad \mathcal{G}=\tilde{\mathcal{G}}.
\end{equation*}

 For the same $\mathcal{J}$ in (\ref{0}), if we have
$$\zeta^2\langle d \mathcal{J}_p(x_1),d \mathcal{J}_p(x_2) \rangle_{ \mathcal{J}(p)}=\langle x_1, x_2 \rangle_p, $$
 then ${\mathcal{S}}$ and $\tilde{{\mathcal{S}}}$ are said to conformal(locally). As in the case of isometry, if in addition to diffeomorphism $\mathcal{J}$ is a bijection, then $\mathcal{J}$ is called conformal globally. In other words, we can say that conformal motion is a composition of dilation
and isometry.  In this case \cite{5}: 
\begin{equation*}\zeta^2\mathcal{E}=\tilde{\mathcal{E}},\quad  \zeta^2\mathcal{F}=\tilde{\mathcal{F}},\quad  \zeta^2\mathcal{G}=\tilde{\mathcal{G}}.\end{equation*}
Here we may call $\mathcal{E},\mathcal{F},\mathcal{G}$ are conformally invariant.

\begin{definition}\label{def1}
Let $\mathfrak{f}:{\mathcal{S}}\rightarrow \tilde{\mathcal{S}}$ be a conformal map between two smooth surfaces, we say that $\mathfrak{f}$ is conformally invariant if $\tilde{\mathfrak{f}}=\zeta^2\mathfrak{f}$ for some dilation factor $\zeta(u,v)$. Similarly, if the same $\mathfrak{f}$ is homothetic, we say that $\mathfrak{f}$ is homothetic invariant if $\tilde{\mathfrak{f}}=c\mathfrak{f}, (c\neq\{0,1\})$.

For example let $K_g$ be the Gaussian curvature of $(\mathcal{S},g)$ and $\chi(\mathcal{S})$ be the Euler characteristic of the surface $\mathcal{S}$. Then according to well known Gauss Bonnet formula:
$$2 \pi \chi (\mathcal{S})= \int_\mathcal{S} K_g ds_g.$$
The above quantity is a topological and a conformal invariant.
\end{definition}

The structure of this paper is as follows. In section $2$, we recall some facts about the curves lying on a smooth surface and give the motivation of the paper. In section $3$, we discuss the main results.

\section{Preliminaries}
Let $\beta:I\subset \mathbb{R}\rightarrow \mathbb{E}^3$ be a smooth curve parameterized by arc length $s$ and $\{\vec{\mathfrak{t}},\vec{\mathfrak{n}},\vec{\mathfrak{b}}\}$ its Serret-Frenet frame. The vectors $\vec{\mathfrak{t}},\vec{\mathfrak{n}}$, and $\vec{\mathfrak{b}}$ are called as the tangent, the normal and the binormal vectors, respectively. The Serret-Frenet equations are given by 
\begin{eqnarray*}
\left\{
\begin{array}{ll}
\vec{\mathfrak{t}}^\prime =\kappa \vec{\mathfrak{n}}\\
\vec{\mathfrak{n}}^\prime = -\kappa \vec{\mathfrak{t}} +\tau \vec{\mathfrak{b}}\\
\vec{\mathfrak{b}}^\prime =-\tau \vec{\mathfrak{n}}.
\end{array}
\right.
\end{eqnarray*}
We call the function $\kappa$ as the curvature of $\beta$ and  $\tau$ as the torsion of $\beta$ satisfying: $\vec{\mathfrak{t}}=\beta^\prime, \vec{\mathfrak{n}}=\frac{{\vec{\mathfrak{t}}}^\prime}{\kappa}$ and $\vec{\mathfrak{b}}=\vec{\mathfrak{t}}\times \vec{\mathfrak{n}}$. At any arbitrary point $\beta(s)$, the plane spanned by $\{\vec{\mathfrak{t}},\vec{\mathfrak{n}}\}$ is called as an osculating plane and the plane spanned by $\{\vec{\mathfrak{t}},\vec{\mathfrak{b}}\}$ is called as a rectifying plane. Similarly, a plane spanned by the vectors $\{\vec{\mathfrak{n}},\vec{\mathfrak{b}}\}$ is called as a normal plane. In other words, the position vector of the curve defines the following curves:
\begin{itemize}
\item If the position vector $\beta(s)$ of the curve $\beta$ lies in the osculating plane then the curve is said to be an osculating curve.
\item If the position vector $\beta(s)$ of the curve $\beta$ lies in the normal plane then the curve is said to be a normal curve.         
\item If the position vector $\beta(s)$ of the curve $\beta$ lies in the rectifying plane then the curve is said to be a rectifying curve.         
\end{itemize}
The classification results of osculating and normal curves are very common which can be found in any standard book of differential geometry of curves and surfaces. After a very long period of time, in 2003 Chen \cite{2} listed a question: When does the position vector of a space curve lie in its rectifying plane? In this paper(\cite{2}), Chen showed that a curve is a rectifying curve if and only if the ratio of the curvature and the torsion is a linear function of arc length $s$. 
For more study, we refer \cite{4,6i}.

The motivation of the present paper starts with a study of Shaikh and Ghosh, where they studied the geometric invariant properties of rectifying curves on a smooth immersed surface under an isometry\cite{8}. Further in \cite{8a}, they investigated the invariant properties of osculating curves under the same motion. Later on, in \cite{9} the authors in \cite{8} and I found the invariant-sufficient condition for a normal curve under an isometric transformation. Afterwards, we generalized the notion of study by the conformal transformation. The invariant properties of rectifying and osculating curves under a conformal transformation are studied in \cite{10,11}. Now, in this paper, we try to investigate the following:

{\it Question:} What are the invariant properties of a normal curve on a smooth immersed surface with respect to a conformal transformation?

 A curve is said to be a normal curve if its position vector field lies in the orthogonal complement of tangent vector i.e., $\beta \cdot \vec{\mathfrak{t}} =0,$ or
\begin{equation}\label{1}
\beta(s)=\nu(s)\vec{\mathfrak{n}}(s)+ \eta(s)\vec{\mathfrak{b}}(s),
\end{equation}
where $\nu,$ $\eta$ are two smooth functions.

Let $\Psi: \Omega(u,v) \subset \mathbb{R}^2\rightarrow \mathcal{S}\subset\mathbb{R}^3 $ be a coordinate chart map of a regular surface ${\mathcal{S}}$. The curve 
$\beta(s)=\beta(u(s),v(s))$ can be thought of a curve 
$ \beta(s)= {\mathcal{S}}(u(s),v(s))$ on the surface ${\mathcal{S}}$. Using the chain rule, we can easily find
\begin{eqnarray}
\nonumber\beta^\prime(s)&=&\Psi_uu^\prime+\Psi_vv^\prime\\
\nonumber\text{or }&&\\
 \nonumber \vec{\mathfrak{t}}(s)&=&\beta^\prime(s)=\Psi_uu^\prime+\Psi_vv^\prime\\
  \nonumber {\vec{\mathfrak{t}}^\prime}(s)&=& u^{\prime\prime}\Psi_u+v^{\prime\prime}\Psi_v+{u^\prime}^2\Psi_{uu}+2u^\prime v^\prime \Psi_{uv}+{v^\prime}^2\Psi_{vv}.
\end{eqnarray}
Now let ${\bf N}$ be the surface normal, we have
\begin{equation}
\label{2}\vec{\mathfrak{n}}(s)=\frac{1}{k(s)}(u''\Psi_u+v''\Psi_v+u'^2\Psi_{uu}+2u'v'\Psi_{uv}+v'^2\Psi_{vv})
\end{equation}
\begin{eqnarray}
\nonumber
\vec{\mathfrak{b}}(s)&=& \vec{\mathfrak{t}}(s)\times \vec{\mathfrak{n}}(s)= \vec{\mathfrak{t}}(s)\times \frac{\vec{\mathfrak{t}}'(s)}{k(s)}\\
\nonumber
&=&\frac{1}{k(s)}\Big[(\Psi_uu'+\Psi_vv')\times(u''\Psi_u+v''\Psi_v+u'^2\Psi_{uu}+2u'v'\Psi_{uv}+v'^2\Psi_{vv})\Big],\\
\nonumber&=&\frac{1}{k(s)}\Big[\{u'v''-u''v'\}{\bf N}+u'^3\Psi_u\times \Psi_{uu}+2u'^2v'\Psi_u\times \Psi_{uv}+u'v'^2\Psi_u\times \Psi_{vv}\\
\label{3} &&+u'^2v'\Psi_v\times \Psi_{uu}+2u'v'^2\Psi_v\times \Psi_{uv}+v'^3\Psi_v\times \Psi_{vv}\Big].
\end{eqnarray}
\begin{definition}
Suppose $\beta$ be a curve with arc length parameterization lying on a surface ${\mathcal{S}}$. This implies that $\mathfrak{t}=\beta^\prime$ is orthogonal to the unit surface normal ${\bf N}$, so $\beta^\prime$, ${\bf N}$ and ${\bf N} \times \beta^\prime$ are mutually orthogonal vectors. Since $\beta$ is of unit speed, we have $\beta^{\prime}\perp \beta^{\prime\prime}$, thus we can write 
\begin{equation*}
\beta^{\prime\prime}=\kappa_{n}{\bf N}+\kappa_g{\bf N}\times \beta^\prime,
\end{equation*}
where $\kappa_n$ is the normal curvature and $\kappa_g$ is the geodesic curvature of $\beta$ and are given by
\begin{eqnarray*}
\left\{
\begin{array}{ll}
\kappa_g =\beta^{\prime \prime}\cdot {\bf N} \times \beta^\prime \\
\kappa_n =\beta^{\prime \prime}\cdot {\bf N}.
\end{array}
\right.
\end{eqnarray*}
Now since we know that $\beta^{\prime \prime}=\kappa(s)\vec{\mathfrak{n}}(s)$, therefore we can write
\begin{equation*}
\kappa_n=\kappa(s)\vec{\mathfrak{n}}(s)\cdot {\bf N}=(u''\Psi_u+v''\Psi_v+u'^2\Psi_{uu}+2u'v'\Psi_{uv}+v'^2\Psi_{vv})\cdot{\bf N}
\end{equation*}
or 
\begin{equation}\label{se1}
\kappa_n = {u^\prime}^2 \mathcal{L} + 2u^\prime v^\prime \mathcal{M} +{v^\prime}^2\mathcal{N},
\end{equation}
where $\mathcal{L},\mathcal{M},\mathcal{N}$ are the second fundamental form coefficients. The curve $\beta$ on $\mathcal{S}$ is called as asymptotic curve if and only if $\kappa_n=0.$
\end{definition}
\section{Conformal image of a normal curve.}
Suppose $\beta(s)$ be a normal curve lying on a smooth immersed surface ${\mathcal{S}}$ in $\mathbb{E}^3$, then with the help of (\ref{1}), (\ref{2}) and (\ref{3}), we can write
\begin{eqnarray}\label{2.1}
\nonumber\beta(s)&=&\frac{\nu(s)}{\kappa(s)}\left[(u^{\prime\prime}\Psi_u+v^{\prime\prime} \Psi_v)+({u^\prime}^2 \Psi_{uu}+2u^\prime v^\prime \Psi_{uv}+{v^\prime}^2 \Psi_{vv})\right]\\
&&+\frac{\eta(s)}{k(s)}\Big[\{u'v''-u''v'\}{\bf N}+u'^3\Psi_u\times \Psi_{uu}+2u'^2v'\Psi_u\times \Psi_{uv}\\
\nonumber
&&+u'v'^2\Psi_u\times \Psi_{vv}+u'^2v'\Psi_v\times \Psi_{uu}+2u'v'^2\Psi_v\times \Psi_{uv}+v'^3\Psi_v\times \Psi_{vv}\Big].
\end{eqnarray}
We shall be considering the expression $\mathcal{J}_*(\beta(s))$ as a product of a $3\times 3$ matrix $J_*$ and a $3\times 1$ matrix $\beta(s)$.
\begin{theorem}
Let $ \mathcal{J}:{\mathcal{S}}\rightarrow {\tilde{\mathcal{S}}}$ be a conformal map between two smooth immersed surfaces ${\mathcal{S}}$ and $\tilde{\mathcal{S}}$ in $\mathbb{E}^3$ and $\beta(s)$ be a normal curve on ${\mathcal{S}}$, then $\tilde{\beta}(s)$ is a normal curve on $\tilde{\mathcal{S}}$ if
\begin{eqnarray}\label{lj2.2}
\nonumber \tilde{\beta}&=&\frac{\nu}{\kappa}\Big[{u^\prime}^2\left(\zeta  \mathcal{J}_\ast\right)_u\Psi_u+{v^\prime}^2\left(\zeta  \mathcal{J}_\ast\right)_v \Psi_v +2u^\prime v^\prime \left(\zeta  \mathcal{J}_\ast\right)_u \Psi_v\Big]+\frac{\eta}{\kappa}\Big[{u^\prime}^3 \zeta  \mathcal{J}_\ast \Psi_u \times \left(\zeta  \mathcal{J}_\ast\right)_u \Psi_u \\
\nonumber &&+2{u^\prime}^2 v^\prime \zeta  \mathcal{J}_\ast \Psi_u \times \left(\zeta  \mathcal{J}_\ast\right)_v \Psi_u +
u^\prime {v^\prime}^2\zeta  \mathcal{J}_\ast \Psi_u \times \left(\zeta  \mathcal{J}_\ast \right)_v \Psi_v+{u^\prime}^2v^\prime \zeta  \mathcal{J}_\ast \Psi_v \times \left(\zeta  \mathcal{J}_\ast\right)_u\Psi_u \\
&&\quad \quad +2u^\prime {v^\prime}^2\zeta  \mathcal{J}_\ast \Psi_v \times \left(\zeta  \mathcal{J}_\ast \right)_u \Psi_v+{v^\prime}^3\zeta  \mathcal{J}_\ast \Psi_v \times \left(\zeta  \mathcal{J}_\ast \right)_v\Psi_v\Big]+\zeta  \mathcal{J}_\ast(\beta).
\end{eqnarray}

\end{theorem}
\begin{proof}
Let $\tilde{\mathcal{S}}$ be the conformal image of ${\mathcal{S}}$ and $\Psi(u,v)$ and $\tilde{\Psi}(u,v)= \mathcal{J}\circ \Psi(u,v) $ be the surface patches of ${\mathcal{S}}$ and $\tilde{{\mathcal{S}}},$ respectively. Then the differential map $d \mathcal{J}= \mathcal{J}_\ast$ of $ \mathcal{J}$ sends each vector of the tangent space $T_p{\mathcal{S}}$ to a dilated tangent vector of the tangent space of $T_{ \mathcal{J}(p)}\tilde{\mathcal{S}}$ with the dilation factor $\zeta$. 
\begin{eqnarray}\label{2.2}
\tilde{\Psi}_u(u,v)&=&\zeta(u,v)  \mathcal{J}_*(\Psi(u,v))\Psi_u\\
\label{2.3}\tilde{\Psi}_v(u,v)&=&\zeta(u,v)  \mathcal{J}_*(\Psi(u,v))\Psi_v.
\end{eqnarray}
Differentiating $(\ref{2.2})$ and $(\ref{2.3})$ partially with respect to both $u$ and $v$ respectively, we get
%\[(*)\qquad\qquad\qquad\qquad\qquad
%\left\{
\begin{eqnarray}\label{q3.7}
\nonumber \tilde{\Psi}_{uu}&=& \zeta_u  \mathcal{J}_\ast \Psi_u +  \zeta\frac{\partial  \mathcal{J}_*}{\partial u}\Psi_u+\zeta  \mathcal{J}_*\Psi_{uu}\\
\tilde{\Psi}_{vv}&=&\zeta_v  \mathcal{J}_\ast \Psi_v+ \zeta\frac{\partial  \mathcal{J}_*}{\partial v}\Psi_v+\zeta  \mathcal{J}_*\Psi_{vv}\\
\nonumber \tilde{\Psi}_{uv}&=& \zeta_u  \mathcal{J}_\ast \Psi_v+ \zeta\frac{\partial  \mathcal{J}_*}{\partial u}\Psi_v+\zeta  \mathcal{J}_*\Psi_{uv}\\
\nonumber &=& \zeta_v  \mathcal{J}_\ast \Psi_u+ \zeta\frac{\partial  \mathcal{J}_*}{\partial v}\Psi_u +\zeta  \mathcal{J}_*\Psi_{uv}.\\\nonumber
 \end{eqnarray}
We can write
\begin{eqnarray}\label{jd2.6}
\nonumber \zeta  \mathcal{J}_\ast \Psi_u \times \left(\zeta_u  \mathcal{J}_\ast \Psi_u+\zeta\frac{\partial  \mathcal{J}_\ast}{\partial u}\Psi_u\right)
 &=&\zeta  \mathcal{J}_\ast \Psi_u \times \left(\zeta_u  \mathcal{J}_\ast \Psi_u+\zeta\frac{\partial  \mathcal{J}_\ast}{\partial u}\Psi_u+\zeta  \mathcal{J}_\ast \Psi_{uu}\right)\\
\nonumber &&-\zeta  \mathcal{J}_\ast (\Psi_u \times \Psi_{uu})\\
&=&\tilde{\Psi}_u \times \tilde{\Psi}_{uu}-\zeta  \mathcal{J}_\ast(\Psi_u \times \Psi_{uu}).
\end{eqnarray} 
Similarly
\begin{eqnarray}\label{j2.7}
\left\{
\begin{array}{ll}
\zeta  \mathcal{J}_\ast \Psi_u \times \left(\zeta_v  \mathcal{J}_\ast \Psi_u+\zeta\frac{\partial  \mathcal{J}_\ast}{\partial v}\Psi_u\right)=\tilde{\Psi}_u \times \tilde{\Psi}_{uv}-\zeta  \mathcal{J}_\ast(\Psi_u \times \Psi_{uv})\vspace{.2cm}\\
\zeta  \mathcal{J}_\ast \Psi_u \times \left(\zeta_v  \mathcal{J}_\ast \Psi_v+\zeta\frac{\partial  \mathcal{J}_\ast}{\partial v}\Psi_v\right)=\tilde{\Psi}_u \times \tilde{\Psi}_{vv}-\zeta  \mathcal{J}_\ast(\Psi_u \times \Psi_{vv})\vspace{.2cm}\\
\zeta  \mathcal{J}_\ast \Psi_v \times \left(\zeta_u  \mathcal{J}_\ast \Psi_u+\zeta\frac{\partial  \mathcal{J}_\ast}{\partial u}\Psi_u\right)=\tilde{\Psi}_v \times \tilde{\Psi}_{uu}-\zeta  \mathcal{J}_\ast(\Psi_v \times \Psi_{uu})\vspace{.2cm}\\
\zeta  \mathcal{J}_\ast \Psi_v \times \left(\zeta_u  \mathcal{J}_\ast \Psi_v+\zeta\frac{\partial  \mathcal{J}_\ast}{\partial u}\Psi_v\right)=\tilde{\Psi}_v \times \tilde{\Psi}_{uv}-\zeta  \mathcal{J}_\ast(\Psi_v \times \Psi_{uv})\vspace{.2cm}\\
\zeta  \mathcal{J}_\ast \Psi_v \times \left(\zeta_v  \mathcal{J}_\ast \Psi_v+\zeta\frac{\partial  \mathcal{J}_\ast}{\partial v}\Psi_v\right)=\tilde{\Psi}_v \times \tilde{\Psi}_{vv}-\zeta  \mathcal{J}_\ast(\Psi_v \times \Psi_{vv}).\vspace{.2cm}\\
\end{array}
\right.
\end{eqnarray} 
 Therefore in view of (\ref{lj2.2}), (\ref{jd2.6}) and (\ref{j2.7}), we have
\begin{eqnarray*}
\nonumber
\tilde{\beta}&=&\frac{\nu}{\kappa}\Big[u^{\prime \prime}\zeta  \mathcal{J}_\ast \Psi_u + v^{\prime\prime}\zeta  \mathcal{J}_\ast \Psi_v+{u^\prime}^2 \left(\zeta_u  \mathcal{J}_\ast \Psi_u +  \zeta\frac{\partial  \mathcal{J}_*}{\partial u}\Psi_u\right)+2{u^\prime} v^\prime \left(\zeta_u  \mathcal{J}_\ast \Psi_v +  \zeta\frac{\partial  \mathcal{J}_*}{\partial u}\Psi_v\right)\\
&&+{v^\prime}^2 \left(\zeta_v  \mathcal{J}_\ast \Psi_v +  \zeta\frac{\partial  \mathcal{J}_*}{\partial v}\Psi_v\right)\Big]+\frac{\eta}{\kappa}\Big[\{u'v''-u''v'\} \mathcal{J}_\ast {\bf N}+u'^3\zeta  \mathcal{J}_\ast (\Psi_u\times \Psi_{uu})\\
&&+2u'^2v'\zeta  \mathcal{J}_\ast ( \Psi_u\times \Psi_{uv})
+u'v'^2\zeta  \mathcal{J}_\ast (\Psi_u\times \Psi_{vv})+u'^2v'\zeta  \mathcal{J}_\ast (\Psi_v\times \Psi_{uu})\\
&&+2u'v'^2\zeta  \mathcal{J}_\ast (\Psi_v\times \Psi_{uv})+v'^3\zeta  \mathcal{J}_\ast( \Psi_v\times \Psi_{vv})\Big]+\frac{\eta}{\kappa}\Big[{u^\prime}^3 \zeta  \mathcal{J}_\ast \Psi_u \times \left(\zeta_u  \mathcal{J}_\ast \Psi_u +  \zeta\frac{\partial  \mathcal{J}_*}{\partial u}\Psi_u\right)\\
&&+2{u^\prime}^2 v^\prime \zeta  \mathcal{J}_\ast \Psi_u \times \left(\zeta_v  \mathcal{J}_\ast \Psi_u +  \zeta\frac{\partial  \mathcal{J}_*}{\partial v}\Psi_u\right)
+u^\prime {v^\prime}^2\zeta  \mathcal{J}_\ast \Psi_u \times \left(\zeta_v  \mathcal{J}_\ast \Psi_v +  \zeta\frac{\partial  \mathcal{J}_*}{\partial v}\Psi_v\right)\\
&&+{u^\prime}^2v^\prime \zeta  \mathcal{J}_\ast \Psi_v \times \left(\zeta_u  \mathcal{J}_\ast \Psi_u +  \zeta\frac{\partial  \mathcal{J}_*}{\partial u}\Psi_u\right)
+2u^\prime {v^\prime}^2\zeta  \mathcal{J}_\ast \Psi_v \times \left(\zeta_u  \mathcal{J}_\ast \Psi_v +  \zeta\frac{\partial  \mathcal{J}_*}{\partial u}\Psi_v\right)\\
&&+{v^\prime}^3\zeta  \mathcal{J}_\ast \Psi_v \times \left(\zeta_v  \mathcal{J}_\ast \Psi_v +  \zeta\frac{\partial  \mathcal{J}_*}{\partial v}\Psi_v\right)\Big]
\end{eqnarray*}
which can be written as
\begin{eqnarray*}
\nonumber\beta(s)&=&\frac{\tilde{\nu}(s)}{\tilde{\kappa}(s)}\left[(u^{\prime\prime}\tilde{\Psi}_u+v^{\prime\prime} \tilde{\Psi}_v)+({u^\prime}^2 \tilde{\Psi}_{uu}+2u^\prime v^\prime \tilde{\Psi}_{uv}+{v^\prime}^2 \tilde{\Psi}_{vv})\right]\\
&&+\frac{\tilde{\eta}(s)}{\tilde{\kappa}(s)}\Big[\{u'v''-u''v'\}\tilde{{\bf N}}+u'^3\Psi_u\times \tilde{\Psi}_{uu}+2u'^2v'\tilde{\Psi}_u\times \tilde{\Psi}_{uv}\\
\nonumber
&&+u'v'^2\tilde{\Psi}_u\times \tilde{\Psi}_{vv}+u'^2v'\tilde{\Psi}_v \times \tilde{\Psi}_{uu}+2u'v'^2\tilde{\Psi}_v\times \tilde{\Psi}_{uv}+v'^3\tilde{\Psi}_v\times \tilde{\Psi}_{vv}\Big]
\end{eqnarray*}
or
\begin{equation*}
\tilde{\beta}(s)=\frac{\tilde{\nu}(s)}{\tilde{\kappa}(s)}\tilde{\vec{\mathfrak{n}}}(s)+\frac{\tilde{\eta}(s)}{\tilde{\kappa}(s)}\tilde{\vec{\mathfrak{b}}}(s)
\end{equation*}
for some $C^\infty$ functions $\tilde{\nu}(s)$ and $\tilde{\eta}(s).$ Here and now onward, we assume that $\frac{\tilde{\nu}}{\tilde{\kappa}}=\frac{\nu}{\kappa}$ and $\frac{\tilde{\eta}}{\tilde{\kappa}}=\frac{\eta}{\kappa}$ . Thus $\tilde{\beta}(s)$ is a normal curve.
\end{proof}
\begin{corollary}
Let $ \mathcal{J}:{\mathcal{S}}\rightarrow \tilde{\mathcal{S}}$ be a homothetic conformal map, where ${\mathcal{S}}$ and $\tilde{\mathcal{S}}$ are smooth surfaces and $\beta(s)$ be a normal curve on ${\mathcal{S}}$. Then $\tilde{\beta}(s)$ is a normal curve on $\tilde{\mathcal{S}}$ if
\begin{eqnarray*}\label{j2.2}
\nonumber \tilde{\beta}&=&\frac{\nu}{\kappa}\Big[{u^\prime}^2c\left(  \mathcal{J}_\ast\right)_u\Psi_u+{v^\prime}^2c\left(  \mathcal{J}_\ast\right)_v \Psi_v +2u^\prime v^\prime c\left(  \mathcal{J}_\ast\right)_u \Psi_v\Big]+\frac{\eta}{\kappa}\Big[{u^\prime}^3 c  \mathcal{J}_\ast \Psi_u \times c\left(  \mathcal{J}_\ast\right)_u \Psi_u \\
&&+2{u^\prime}^2 v^\prime c  \mathcal{J}_\ast \Psi_u \times c\left(  \mathcal{J}_\ast\right)_v \Psi_u +
u^\prime {v^\prime}^2c  \mathcal{J}_\ast \Psi_u \times c\left(  \mathcal{J}_\ast \right)_v \Psi_v+{u^\prime}^2v^\prime c  \mathcal{J}_\ast \Psi_v \times c\left(  \mathcal{J}_\ast\right)_u\Psi_u \\
&&+2u^\prime {v^\prime}^2c  \mathcal{J}_\ast \Psi_v \times c\left(  \mathcal{J}_\ast \right)_u \Psi_v+{v^\prime}^3c  \mathcal{J}_\ast \Psi_v \times c\left(  \mathcal{J}_\ast \right)_v\Psi_v\Big]+c  \mathcal{J}_\ast(\beta).
\end{eqnarray*}
\end{corollary}
\begin{proof}
In case of a homothetic map the dilation function $\zeta(u,v)=c \neq \{0,1\}$. Substituting in (\ref{lj2.2}), we get the above expression.  
\end{proof}
\begin{corollary}\cite{9}
Let $ \mathcal{J}:{\mathcal{S}}\rightarrow \tilde{\mathcal{S}}$ be an isometry, where ${\mathcal{S}}$ and $\tilde{\mathcal{S}}$ are smooth surfaces and $\beta(s)$ be a normal curve on ${\mathcal{S}}$. Then $\tilde{\beta}(s)$ is a normal curve on $\tilde{\mathcal{S}}$ if
\begin{eqnarray*}\label{j2.2}
\nonumber \tilde{\beta}&=&\frac{\nu}{\kappa}\Big[{u^\prime}^2\frac{\partial  \mathcal{J}_\ast}{\partial u}\Psi_u+{v^\prime}^2\frac{\partial  \mathcal{J}_\ast}{\partial v} \Psi_v +2u^\prime v^\prime \frac{\partial  \mathcal{J}_\ast}{\partial u} \Psi_v\Big]+\frac{\eta}{\kappa}\Big[{u^\prime}^3   \mathcal{J}_\ast \Psi_u \times \frac{\partial  \mathcal{J}_\ast}{\partial u} \Psi_u \\
&&+2{u^\prime}^2 v^\prime   \mathcal{J}_\ast \Psi_u \times \frac{\partial  \mathcal{J}_\ast}{\partial v} \Psi_u +
u^\prime {v^\prime}^2  \mathcal{J}_\ast \Psi_u \times \frac{\partial  \mathcal{J}_\ast}{\partial v} \Psi_v+{u^\prime}^2v^\prime   \mathcal{J}_\ast \Psi_v \times \frac{\partial  \mathcal{J}_\ast}{\partial u}\Psi_u \\
&&+2u^\prime {v^\prime}^2  \mathcal{J}_\ast \Psi_v \times \frac{\partial  \mathcal{J}_\ast}{\partial u} \Psi_v+{v^\prime}^3  \mathcal{J}_\ast \Psi_v \times \frac{\partial  \mathcal{J}_\ast}{\partial v}\Psi_v\Big]+  \mathcal{J}_\ast(\beta).
\end{eqnarray*}
\end{corollary}
\begin{proof}
A conformal transformation is the composition of a dilation function and an isometry. Substituting $\zeta=1$ in (\ref{lj2.2}), we get the above expression.  
\end{proof}
\begin{theorem}
Let ${\mathcal{S}}$ and $\tilde{\mathcal{S}}$ be two conformal smooth surfaces and $\beta(s)$ be a normal curve on ${\mathcal{S}}$. Then for the normal component along the surface normal, we have
\begin{equation}\label{x8}
\tilde{\beta}\cdot {\tilde{\bf N}}-\zeta^4 \beta \cdot {\bf N}=\frac{\nu}{\kappa}(\tilde{\kappa}_n-\zeta^4 \kappa_n)+\frac{\eta}{\kappa}h(\mathcal{E},\mathcal{G},\mathcal{F},\zeta),
\end{equation}
where 
\begin{equation}\label{q9}
h(\mathcal{E},\mathcal{G},\mathcal{F},\zeta)= \Big[{u^\prime}^3 \theta_{11}^2 -{v^\prime}^3 \theta_{22}^1+2{u^\prime}^2v^\prime \theta_{12}^2\\
 +u^\prime {v^\prime}^2 \theta_{22}^2-{u^\prime}^2v^\prime \theta_{11}^1+2u^\prime {v^\prime}^2 \theta_{12}^1\Big] W^2.
\end{equation}
\end{theorem}
\begin{proof}
Let $\tilde{\mathcal{S}}$ be the conformal image of ${\mathcal{S}}$ and $\Psi(u,v)$ and $\tilde{\Psi}(u,v)= \mathcal{J}\circ \Psi(u,v) $ be the surface patches of ${\mathcal{S}}$ and $\tilde{{\mathcal{S}}},$ respectively. We can easily find
\begin{eqnarray*}
\beta \cdot {\bf N} &=&\frac{\nu}{\kappa}\Big[{u^\prime}^2 \Psi_{uu}\cdot (\Psi_u \times \Psi_v)+{v^\prime}^2 \Psi_{vv}\cdot (\Psi_u \times \Psi_v)+2u^\prime v^\prime \Psi_{uv}\cdot (\Psi_u \times \Psi_v)\Big]\\
&&+\frac{\eta}{\kappa}\Big[(u^\prime v^{\prime\prime}-v^\prime u^{\prime\prime})(\mathcal{E}\mathcal{G}-\mathcal{F}^2)+{u^\prime}^3(\Psi_u \times \Psi_{uu})\cdot (\Psi_u \times \Psi_v)\\
&&+2{u^\prime}^2v^\prime (\Psi_u \times \Psi_{uv})\cdot (\Psi_u \times \Psi_v)
+{v^\prime}^2u^\prime (\Psi_u \times \Psi_{vv})\cdot (\Psi_u \times \Psi_v)\\
&&+{u^\prime}^2v^\prime (\Psi_v \times \Psi_{uu})\cdot (\Psi_u \times \Psi_v)
+2{u^\prime}{v^\prime}^2 (\Psi_v \times \Psi_{uv})\cdot (\Psi_u \times \Psi_v)\\
&&+{v^\prime}^3 (\Psi_v \times \Psi_{vv})\cdot (\Psi_u \times \Psi_v)\Big]
\end{eqnarray*}
or
\begin{eqnarray*}
\beta \cdot {\bf N} &=&\frac{\nu}{\kappa}\Big[{u^\prime}^2 \mathcal{L}+{v^\prime}^2 \mathcal{N}+2u^\prime v^\prime \mathcal{M}\Big]
+\frac{\eta}{\kappa}\Big[(u^\prime v^{\prime\prime}-v^\prime u^{\prime\prime})(\mathcal{E}\mathcal{G}-\mathcal{F}^2)
+{u^\prime}^3\{\mathcal{E}(\Psi_{uu}\cdot \Psi_v)\\
&&-\mathcal{F}(\Psi_{uu}\cdot\Psi_u)\}+2{u^\prime}^2 v^\prime\{\mathcal{E}(\Psi_{uv}\cdot \Psi_v)-\mathcal{F}(\Psi_{uv}\cdot\Psi_u)\}+{u^\prime}{v^\prime}^2\{\mathcal{E}(\Psi_{vv}\cdot \Psi_v)\\
&&-\mathcal{F}(\Psi_{vv}\cdot\Psi_u)\}+{u^\prime}^2v^\prime\{\mathcal{F}(\Psi_{uu}\cdot \Psi_v)-\mathcal{G}(\Psi_{uu}\cdot\Psi_u)\}+2{u^\prime}{v^\prime}^2\{\mathcal{F}(\Psi_{uv}\cdot \Psi_v)\\
&&-\mathcal{G}(\Psi_{uv}\cdot\Psi_u)\}+{v^\prime}^3\{\mathcal{F}(\Psi_{vv}\cdot \Psi_v)-\mathcal{G}(\Psi_{vv}\cdot\Psi_u)\}\Big].
\end{eqnarray*}
We know that $\mathcal{E}_u= (\Psi_u \cdot \Psi_u)_u=\frac{1}{2}\Psi_{uu}\cdot\Psi_u,$ or 
\begin{equation}\label{1.16}
\Psi_{uu}\cdot\Psi_u=\frac{\mathcal{E}_u}{2}.
\end{equation}
On the similar lines, we can find
\begin{eqnarray}\label{1.17}
\left\{
\begin{array}{ll}
\Psi_{uu} \cdot \Psi_v =\mathcal{F}_u-\frac{\mathcal{E}_v}{2}, \Psi_{vv}\cdot \Psi_v =\frac{\mathcal{G}_v}{2},\Psi_{vv}\cdot \Psi_u= \mathcal{F}_v-\frac{\mathcal{G}_u}{2},\\
\Psi_{uv}\cdot \Psi_v =\frac{\mathcal{G}_u}{2}, \Psi_{uv}\cdot \Psi_u =\frac{\mathcal{E}_v}{2}.
\end{array}
\right.
\end{eqnarray}
Therefore in view of (\ref{1.16}) and (\ref{1.17}), $\beta \cdot {\bf N}$ turns out to be
\begin{eqnarray*}
\beta \cdot {\bf N} &=&\frac{\nu}{\kappa}\Big[{u^\prime}^2 \mathcal{L}+{v^\prime}^2 \mathcal{N}+2u^\prime v^\prime \mathcal{M}\Big]\\
&&+\frac{\eta}{\kappa}\Big[(u^\prime v^{\prime\prime}-v^\prime u^{\prime\prime})(\mathcal{E}\mathcal{G}-\mathcal{F}^2)
+{u^\prime}^3\left \{\mathcal{E}\left(\mathcal{F}_u-\frac{\mathcal{E}_v}{2}\right)-\frac{\mathcal{F}\mathcal{E}_u}{2}\right \}\\
&&+2{u^\prime}^2 v^\prime \left \{\frac{\mathcal{E}\mathcal{G}_u}{2}-\frac{\mathcal{F}\mathcal{E}_v}{2}\right \}
+{u^\prime}{v^\prime}^2 \left \{\frac{\mathcal{E}\mathcal{G}_v}{2}-\mathcal{F}\left(\mathcal{F}_v-\frac{\mathcal{G}_u}{2}\right)\right\}\\
&&+{u^\prime}^2v^\prime \left \{\mathcal{F}\left(\mathcal{F}_u-\frac{\mathcal{E}_v}{2}\right)-\frac{\mathcal{G}\mathcal{E}_u}{2}\right \}+2{u^\prime}{v^\prime}^2 \left \{\frac{\mathcal{F}\mathcal{G}_u}{2} - \frac{\mathcal{G}\mathcal{E}_v}{2} \right\}\\
&&+{v^\prime}^3 \left \{\frac{\mathcal{F}\mathcal{G}_v}{2}-\mathcal{G}\left(\mathcal{F}_v-\frac{\mathcal{G}_u}{2}\right)\right \}\Big]
\end{eqnarray*}
or
\begin{eqnarray}\nonumber
\beta \cdot {\bf N}&=&\frac{\nu}{\kappa}\kappa_n+\frac{\eta}{\kappa}(\mathcal{E}\mathcal{G}-\mathcal{F}^2)\Big[(u^\prime v^{\prime \prime}-v^\prime u^{\prime \prime})+{u^\prime}^3\mathtt{\Gamma}_{11}^2 -{v^\prime}^3 \mathtt{\Gamma}_{22}^1+2{u^\prime}^2v^\prime \mathtt{\Gamma}_{12}^2\\
\label{x13}&& +u^\prime {v^\prime}^2 \mathtt{\Gamma}_{22}^2-{u^\prime}^2v^\prime \mathtt{\Gamma}_{11}^1+2u^\prime {v^\prime}^2\mathtt{\Gamma}_{12}^1\Big],
\end{eqnarray}
where $\mathtt{\Gamma}_{ij}^k,(i,j,k=1,2)$ are Christoffel symbols of second kind given by:
\begin{equation}\label{q12}\left\{
\begin{array}{ll}
\mathtt{\Gamma}_{11}^1=\frac{1}{2W^2}\left\{\mathcal{G}\mathcal{E}_u+\mathcal{F}[\mathcal{E}_v-2\mathcal{F}_u]\right\}, \quad \mathtt{\Gamma}_{22}^2=\frac{1}{2W^2}\left\{\mathcal{E}\mathcal{G}_v+\mathcal{F}[\mathcal{G}_v-2\mathcal{F}_v]\right\}\\
\mathtt{\Gamma}_{11}^2=\frac{1}{2W^2}\left\{\mathcal{E}[2\mathcal{F}_u-\mathcal{E}_v]-\mathcal{F}\mathcal{E}_v\right\},\quad \mathtt{\Gamma}_{22}^1=\frac{1}{2W^2}\left\{\mathcal{G}[2\mathcal{F}_v-\mathcal{G}_u]-\mathcal{F}\mathcal{G}_v\right\}\\
\mathtt{\Gamma}_{12}^2=\frac{1}{2W^2}\left\{\mathcal{E}\mathcal{G}_u-\mathcal{F}\mathcal{E}_v\right\}=\mathtt{\Gamma}_{21}^2,\quad
\mathtt{\Gamma}_{21}^1=\frac{1}{2W^2}\left\{\mathcal{G}\mathcal{E}_v-\mathcal{F}\mathcal{G}_u\right\}=\mathtt{\Gamma}_{12}^1
\end{array}
\right.
\end{equation}
and $W=\sqrt{\mathcal{E}\mathcal{G}-\mathcal{F}^2}$.

Under conformal motion, we have
\begin{equation}\label{z2.8}
\zeta^2\mathcal{E}=\tilde{\mathcal{E}},\quad  \zeta^2\mathcal{F}=\tilde{\mathcal{F}},\quad  \zeta^2\mathcal{G}=\tilde{\mathcal{G}}.
\end{equation}
This implies that
\begin{eqnarray}\label{12a}
\left\{
\begin{array}{ll}
\tilde{\mathcal{E}}_u=2\zeta \zeta_u \mathcal{E} + \zeta^2 \mathcal{E}_u, \quad \tilde{\mathcal{E}}_v=2\zeta \zeta_v \mathcal{E} + \zeta^2 \mathcal{E}_v,\\
\tilde{\mathcal{F}}_u=2\zeta \zeta_u \mathcal{F} + \zeta^2 \mathcal{F}_u, \quad \tilde{\mathcal{F}}_v=2\zeta \zeta_v \mathcal{F} + \zeta^2 \mathcal{F}_v,\\
\tilde{\mathcal{G}}_u=2\zeta \zeta_u \mathcal{G} + \zeta^2 \mathcal{G}_u, \quad \tilde{\mathcal{G}}_v=2\zeta \zeta_v \mathcal{G} + \zeta^2 \mathcal{G}_v.
\end{array}
\right.
\end{eqnarray}
After the conformal motion, the Christoffel symbols turn out to be
\begin{eqnarray}\label{q13}
\begin{array}{ll}
\tilde{\mathtt{\Gamma}}_{11}^1=\mathtt{\Gamma}_{11}^1 +\theta_{11}^1,\quad \tilde{\mathtt{\Gamma}}_{11}^2=\mathtt{\Gamma}_{11}^2 +\theta_{11}^2, \quad \tilde{\mathtt{\Gamma}}_{12}^1=\mathtt{\Gamma}_{12}^1 +\theta_{12}^1,\\
\tilde{\mathtt{\Gamma}}_{12}^2=\mathtt{\Gamma}_{12}^2 +\theta_{12}^2,\quad \tilde{\mathtt{\Gamma}}_{22}^1=\mathtt{\Gamma}_{22}^1 +\theta_{22}^1, \quad \tilde{\mathtt{\Gamma}}_{22}^2=\mathtt{\Gamma}_{22}^2 +\theta_{22}^2,
\end{array}
\end{eqnarray}
where
\begin{eqnarray}\label{q16}
\left\{\begin{array}{ll}
\theta_{11}^1=
\frac{\mathcal{E}\mathcal{G}\zeta_u -2\mathcal{F}^2\zeta_u +\mathcal{F}\mathcal{E}\zeta_v}{\zeta W^2},\quad \theta_{11}^2=
\frac{\mathcal{E}\mathcal{F}\zeta_u -\mathcal{E}^2 \zeta_v}{\zeta W^2},\vspace{.1cm}\\
\theta_{12}^1=\frac{\mathcal{E}\mathcal{G}\zeta_v-\mathcal{F}\mathcal{G}\zeta_u}{\zeta W^2},
\quad \theta_{12}^2=\frac{\mathcal{E}\mathcal{G}\zeta_u - \mathcal{F}\mathcal{E} \zeta_v}{\zeta W^2},\vspace{.1cm}\\
\theta_{22}^1=\frac{\mathcal{G}\mathcal{F}\zeta_v - \mathcal{G}^2 \zeta_u}{\zeta W^2},\quad \theta_{22}^2=\frac{\mathcal{E}\mathcal{G}\zeta_v -2\mathcal{F}^2 \zeta_v +\mathcal{F}\mathcal{G}\zeta_u}{\zeta W^2}.
\end{array}\right.
\end{eqnarray}
Now if $\beta$ is a normal curve on $\tilde{\mathcal{S}}$, in view of (\ref{x13}), (\ref{q13}) and (\ref{q16}), we get
\begin{equation}
\tilde{\beta}\cdot {\tilde{\bf N}}-\zeta^4 \beta \cdot {\bf N}=\frac{\nu}{\kappa}(\tilde{\kappa}_n-\zeta^4 \kappa_n)+\frac{\eta}{\kappa}\Big[{u^\prime}^3 \theta_{11}^2 -{v^\prime}^3 \theta_{22}^1+2{u^\prime}^2v^\prime \theta_{12}^2\\
 +u^\prime {v^\prime}^2 \theta_{22}^2-{u^\prime}^2v^\prime \theta_{11}^1+2u^\prime {v^\prime}^2 \theta_{12}^1\Big] W^2.
\end{equation}
This proves the claim.
\end{proof}
\begin{corollary}
Let ${\mathcal{S}}$ and $\tilde{\mathcal{S}}$ be two homothetic conformal smooth surfaces and $\beta(s)$ be a normal curve on ${\mathcal{S}}$. Then  for the normal component along the surface normal, we have
\begin{equation}\label{x3.19}
\tilde{\beta}\cdot {\tilde{\bf N}}-c^4 \left(\beta \cdot {\bf N}\right)=\frac{\nu}{\kappa}(\tilde{\kappa}_n-c^4 \kappa_n).
\end{equation}
Moreover, this normal component is conformally invariant if the position vector of $\beta$ is in the binormal direction or the normal curvature is conformally invariant.
\end{corollary}
\begin{proof}
Letting $\zeta(u,v)=c$, from (\ref{x8}), (\ref{q9}) and (\ref{q16}), the claim in (\ref{x3.19}) is straightforward. 

Again from (\ref{x3.19}), we see that $\beta$ is conformally invariant if and only if $\nu=0$, i.e., $\beta(s)=\eta(s)b(s)$ or $\tilde{\kappa}_n=c^4 \kappa_n$.

\end{proof}
\begin{corollary}\cite{9}
Let ${\mathcal{S}}$ and $\tilde{\mathcal{S}}$ be two isometric smooth surfaces and $\beta(s)$ be a normal curve on ${\mathcal{S}}$. Then for the normal component of $\beta(s)$ along the surface normal, we have 
\begin{equation*}
\tilde{\beta}\cdot {\tilde{\bf N}}- \left(\beta \cdot {\bf N}\right)=\frac{\nu}{\kappa}(\tilde{\kappa}_n- \kappa_n).
\end{equation*} 
Moreover under such an isometry the normal component along the surface normal is invariant if the position vector of $\beta$ is in the binormal direction or the normal curvature is invariant.
\end{corollary}
\begin{remark}
Let $ \mathcal{J}: {\mathcal{S}} \rightarrow \tilde{\mathcal{S}}$ be an isometry, then the dilation factor of conformality is $\zeta=1$. From (\ref{q13}) and (\ref{q16}), it is straightforward to check $\tilde{\mathtt{\Gamma}}_{ij}^k=\mathtt{\Gamma}_{ij}^k,(i,j,k=1,2),$ i.e., Christoffel symbols are invariant under isometry. 
\end{remark}
\begin{theorem}
Let ${\mathcal{S}}$ and $\tilde{\mathcal{S}}$ be two conformal smooth surfaces and $\beta(s)$ be a normal curve on ${\mathcal{S}}$. Then for the tangential component, we have
\begin{equation}\label{ik1}
\tilde{\beta}\cdot{ \tilde{\bf T}}-\zeta^2\left(\beta\cdot {\bf T}\right) =(ag_1+bg_2)+\frac{\eta}{\kappa}\left(\tilde{\kappa}_n-\zeta^2 \kappa_n\right)(a v^\prime+b u^\prime),
\end{equation}
where $g_1$ and $g_2$ are given by (\ref{y21a}) and (\ref{y21b}), respectively.
\end{theorem}
\begin{proof}
From (\ref{2.1}), we see that
\begin{eqnarray*}
\nonumber  \beta \cdot  \Psi_u &=&\frac{\nu}{\kappa}\Big[u^{\prime \prime} \mathcal{E}  +v^{\prime \prime} \mathcal{F}+{u^\prime}^2 \Psi_{uu} \cdot  \Psi_u + 2u^\prime v^\prime  \Psi_{uv} \cdot  \Psi_u  + {v^\prime}^2 \Psi_{vv} \cdot  \Psi_u  \Big]\\
\label{a1.13}&&+\frac{\eta}{\kappa}\Big[{u^\prime}^2 v^\prime \mathcal{L}  +2{v^\prime}^2 u^\prime  \mathcal{M}+{v^\prime}^3  \mathcal{N} \Big]
\end{eqnarray*}
or by using (\ref{1.16}), (\ref{1.17}) and (\ref{se1}), we can write the above equation as
\begin{eqnarray*}
\nonumber  \beta \cdot  \Psi_u &=&\frac{\nu}{\kappa}\Big[u^{\prime \prime} \mathcal{E}  +v^{\prime \prime} \mathcal{F}+{u^\prime}^2 \frac{\mathcal{E}_u}{2} + 2u^\prime v^\prime  \frac{\mathcal{E}_v}{2}  + {v^\prime}^2 \left(\mathcal{F}_v-\frac{\mathcal{G}_u}{2}\right)  \Big]+\frac{\eta}{\kappa}v^\prime \kappa_n.
\end{eqnarray*}
Now if $\tilde{\beta}$ be the conformal image of $\beta$ on $\tilde{\mathcal{S}}$, we have
\begin{eqnarray*}
\nonumber  \tilde{\beta} \cdot  \tilde{\Psi}_u &=&\frac{\tilde{\nu}}{\tilde{\kappa}}\Big[u^{\prime \prime} \tilde{\mathcal{E}}  +v^{\prime \prime} \tilde{\mathcal{F}}+{u^\prime}^2 \frac{\tilde{\mathcal{E}}_u}{2} + 2u^\prime v^\prime  \frac{\tilde{\mathcal{E}}_v}{2}  + {v^\prime}^2 \left(\tilde{\mathcal{F}}_v-\frac{\tilde{\mathcal{G}}_u}{2}\right)  \Big]+\frac{\tilde{\eta}}{\tilde{\kappa}}v^\prime \tilde{\kappa}_n.
\end{eqnarray*}
In view of (\ref{z2.8}) and (\ref{12a}), the above equation turns out to be
\begin{eqnarray*}
\tilde{\beta} \cdot \tilde{\Psi}_u&=&\frac{\nu}{\kappa}\Big[u^{\prime \prime}\zeta^2 \mathcal{E} +v^{\prime\prime}\zeta^2 \mathcal{F}+{u^\prime}^2\frac{(2\zeta \zeta_u \mathcal{E} + \zeta^2 \mathcal{E}_u)}{2} +u^\prime v^\prime(2\zeta \zeta_v \mathcal{E} + \zeta^2 \mathcal{E}_v)\\
&&+{v^\prime}^2\left(2\zeta \zeta_v \mathcal{F} + \zeta^2 \mathcal{F}_v-\frac{2\zeta \zeta_u \mathcal{G} + \zeta^2 \mathcal{G}_u}{2}\right) \Big]+\frac{\eta}{\kappa}v^\prime \tilde{\kappa}_n
\end{eqnarray*}
or 
\begin{equation}\label{y20}
\tilde{\beta}\cdot \tilde{\Psi}_u-\zeta^2 (\beta \cdot \Psi_u)=g_1(\mathcal{E},\mathcal{F},\mathcal{G},\zeta)+\frac{\eta}{\kappa}v^\prime \left(\tilde{\kappa}_n-\zeta^2 \kappa_n\right),
\end{equation}
where
\begin{equation}\label{y21a}
g_1(\mathcal{E},\mathcal{F},\mathcal{G},\zeta)=\frac{\nu}{\kappa}\Big[{u^\prime}^2\zeta \zeta_u \mathcal{E} +2u^\prime v^\prime \zeta \zeta_v \mathcal{E} +{v^\prime}^2\left(2\zeta \zeta_v \mathcal{F}-\zeta \zeta_u \mathcal{G}\right)\Big].
\end{equation}
On the similar lines, it is easy to find
\begin{equation}\label{y21}
\tilde{\beta}\cdot \tilde{\Psi}_v-\zeta^2 (\beta \cdot \Psi_v)=g_2(\mathcal{E},\mathcal{F},\mathcal{G},\zeta)+\frac{\eta}{\kappa}u^\prime \left(\tilde{\kappa}_n-\zeta^2 \kappa_n\right),
\end{equation}
where
\begin{equation}\label{y21b}
g_2(\mathcal{E},\mathcal{F},\mathcal{G},\zeta)=\frac{\nu}{\kappa}\Big[{u^\prime}^2\left(2 \zeta \zeta_u \mathcal{F}- \zeta \zeta_v \mathcal{E} \right)+2u^\prime v^\prime \zeta \zeta_u \mathcal{G}+{v^\prime}^2 \zeta \zeta_v \mathcal{G}\Big].
\end{equation}
Now with the help of $(\ref{y20})$ and $(\ref{y21})$, we get
\begin{eqnarray*}
\tilde{\beta}\cdot{\tilde{\bf T}}-\zeta^2\left(\beta\cdot {\bf T}\right) &=&\tilde{\beta}\cdot(a\tilde{\Psi}_u+b\tilde{\Psi}_v)-\zeta^2\beta\cdot(a\Psi
_u+b\Psi
_v)\\
&=& a(\tilde{\beta}\cdot \tilde{\Psi}_u -\zeta^2\beta\cdot \Psi_u)+b(\tilde{\beta}\cdot \tilde{\Psi}_v -\zeta^2\beta\cdot \Psi_v)\\
&=& a\left\{(g_1+\frac{\eta}{\kappa}v^\prime \left(\tilde{\kappa}_n-\zeta^2 \kappa_n\right)\right\}+b\left\{ g_2+\frac{\eta}{\kappa}u^\prime \left(\tilde{\kappa}_n-\zeta^2 \kappa_n\right)  \right\}\\
&=&(ag_1+bg_2)+\frac{\eta}{\kappa}\left(\tilde{\kappa}_n-\zeta^2 \kappa_n\right)(a v^\prime+b u^\prime).
\end{eqnarray*}
\end{proof}
\begin{corollary}Let $ \mathcal{J}:{\mathcal{S}}\rightarrow \tilde{\mathcal{S}}$ be a conformal homothetic map and $\beta$ be a normal curve on $\mathcal{S}$. The the tangential component of $\beta$ is homothetic invariant if and only if the position vector of $\beta$ is in the normal direction or the normal curvature is homothetic invariant. 
\end{corollary}
\begin{proof}
For a homothetic conformal map, from (\ref{ik1}), we have
\begin{equation*}
\tilde{\beta}\cdot{ \tilde{\bf T}}-c^2\left(\beta\cdot {\bf T}\right) =\frac{\eta}{\kappa}\left(\tilde{\kappa}_n-c^2 \kappa_n\right)(a v^\prime+b u^\prime).
\end{equation*}
The conclusions are straightforward from the above expression.
\end{proof}
\begin{corollary}\cite{9}Let $ \mathcal{J}:{\mathcal{S}}\rightarrow \tilde{\mathcal{S}}$ be an isometry and $\beta$ be a normal curve on $\mathcal{S}$. The  for the tangential component of $\beta$, we have 
\begin{equation*}
\tilde{\beta}\cdot{\tilde{\bf T}}-\left(\beta\cdot {\bf T}\right) =\frac{\eta}{\kappa}\left(\tilde{\kappa}_n- \kappa_n\right)(a v^\prime+b u^\prime)
\end{equation*}
and is invariant if and only if the position vector of $\beta$ is in the normal direction or the normal curvature is invariant. 
\end{corollary}
\begin{proposition}
Let $ \mathcal{J}$ be a conformal map between two smooth surfaces ${\mathcal{S}}$ and $\tilde{\mathcal{S}}$ and let $\beta(s)$ be a parameterized curve on ${\mathcal{S}}$ such that $\tilde{\beta}(s)= \mathcal{J}\circ \beta(s)$ is conformal parameterized image of $\beta$ on $\tilde{\mathcal{S}}$. Then for the geodesic curvature of $\beta$,  we have
\begin{equation}\label{i20}
 \tilde{\kappa}_g-\zeta^2 \kappa_g= f(\mathcal{E},\mathcal{F},\mathcal{G},\zeta).
 \end{equation}
\end{proposition}
\begin{proof}
Let $\beta$ be a parameterized curve on a smooth surface ${\mathcal{S}}$, then the geodesic curvature is given by Beltrami formula as:
\begin{equation}\label{3.18}
\kappa_g=\Big[\mathtt{\Gamma}_{11}^2{u^\prime}^3+(2\mathtt{\Gamma}_{12}^2-\mathtt{\Gamma}_{11}^1){u^\prime}^2 v^\prime +(\mathtt{\Gamma}_{22}^2-2\mathtt{\Gamma}_{12}^1)u^\prime {v^\prime}^2-\mathtt{\Gamma}_{22}^1{v^\prime}^3+u^\prime v^{\prime\prime}-u^{\prime \prime}v^\prime\Big] \sqrt{\mathcal{E}\mathcal{G}-\mathcal{F}^2}.
\end{equation}
Now, let $\tilde{\beta}= \mathcal{J}\circ \beta$ be the conformal image of $\beta$ on $\tilde{\mathcal{S}}$, then with the help of (\ref{q13}), we have
 \begin{eqnarray*}
 \tilde{\kappa}_g&=& \Big[\mathtt{\Gamma}_{11}^2{u^\prime}^3+(2\mathtt{\Gamma}_{12}^2-\mathtt{\Gamma}_{11}^1){u^\prime}^2 v^\prime +(\mathtt{\Gamma}_{22}^2-2\mathtt{\Gamma}_{12}^1)u^\prime {v^\prime}^2-\mathtt{\Gamma}_{22}^1{v^\prime}^3+u^\prime v^{\prime\prime}-u^{\prime \prime}v^\prime\Big] W^2\\
 &&+\Big[\theta_{11}^2{u^\prime}^3+(2\theta_{12}^2-\theta_{11}^1){u^\prime}^2 v^\prime +(\theta_{22}^2-2\theta_{12}^1)u^\prime {v^\prime}^2-\theta_{22}^1{v^\prime}^3\Big]W^2
 \end{eqnarray*}
 or 
 \begin{equation*}
 \tilde{\kappa}_g-\zeta^2 \kappa_g= f(\mathcal{E},\mathcal{F},\mathcal{G},\zeta),
 \end{equation*}
 where $f(\mathcal{E},\mathcal{G},\mathcal{F},\zeta)= \left\{\theta_{11}^2{u^\prime}^3+(2\theta_{12}^2-\theta_{11}^1){u^\prime}^2 v^\prime +(\theta_{22}^2-2\theta_{12}^1)u^\prime {v^\prime}^2-\theta_{22}^1{v^\prime}^3\right\}W^2.$
 
\noindent This proves the claim.
\end{proof}
{\bf Note}: It is to be noted that, in particular, if $\beta$ is a normal curve and $ \mathcal{J}$ is isometry(or homothetic), from (\ref{i20}) we see that $\kappa_g$ is invariant(or homothetic invariant). 

{\bf Acknowledgment:} I am very thankful to Prof. Absos A. Shaikh for his valuable suggestions.

\end{document}